\documentclass[11 pt]{amsart}
\usepackage{amsmath, amsthm,amssymb,bbm,enumerate,mathrsfs,mathtools}
\usepackage{a4wide}
\usepackage{tikz,xcolor,verbatim}
\usetikzlibrary{calc,shapes}
\usepackage{hyperref}
\usepackage{makecell}
\usepackage[numbers,sort&compress]{natbib}

 \usepackage[yyyymmdd,hhmmss]{datetime}
 \usepackage{background}
 \SetBgContents{ver. \currenttime \qquad \today }
 \SetBgScale{1}
 \SetBgAngle{0}
 \SetBgOpacity{1}
 \SetBgPosition{current page.south west} 
 \SetBgHshift{3cm}
 \SetBgVshift{0.5cm} 

\tikzset{
  bigblue/.style={circle, draw=blue!80,fill=blue!40,thick, inner sep=1.5pt, minimum size=5mm},
  bigred/.style={circle, draw=red!80,fill=red!40,thick, inner sep=1.5pt, minimum size=5mm},
  bigblack/.style={circle, draw=black!100,fill=black!40,thick, inner sep=1.5pt, minimum size=5mm},
  bluevertex/.style={circle, draw=blue!100,fill=blue!100,thick, inner sep=0pt, minimum size=2mm},
  redvertex/.style={circle, draw=red!100,fill=red!100,thick, inner sep=0pt, minimum size=2mm},
  blackvertex/.style={circle, draw=black!100,fill=black!100,thick, inner sep=0pt, minimum size=1.5mm},  
  whitevertex/.style={circle, draw=black!100,fill=white!100,thick, inner sep=0pt, minimum size=2mm},  
  smallblack/.style={circle, draw=black!100,fill=black!100,thick, inner sep=0pt, minimum size=1mm},
  smallwhite/.style={circle, draw=black!100,fill=white!100,thick, inner sep=0pt, minimum size=1mm}, 
  dummywhite/.style={circle, draw=white!100,fill=white!100,thick, inner sep=0pt, minimum size=.2mm}
}

\makeatletter
\pgfdeclareshape{myNode}{
  \inheritsavedanchors[from=rectangle] 
  \inheritanchorborder[from=rectangle]
  \inheritanchor[from=rectangle]{center}
  \inheritanchor[from=rectangle]{north}
  \inheritanchor[from=rectangle]{south}
  \inheritanchor[from=rectangle]{west}
  \inheritanchor[from=rectangle]{east}
  \backgroundpath{
    \southwest \pgf@xa=\pgf@x \pgf@ya=\pgf@y
    \northeast \pgf@xb=\pgf@x \pgf@yb=\pgf@y
    \pgfsetcornersarced{\pgfpoint{5pt}{5pt}}
    \pgfpathmoveto{\pgfpoint{\pgf@xa}{\pgf@ya}}
    \pgfpathlineto{\pgfpoint{\pgf@xa}{\pgf@yb}}
    \pgfpathlineto{\pgfpoint{\pgf@xb}{\pgf@yb}}
    \pgfsetcornersarced{\pgfpoint{5pt}{5pt}}
    \pgfpathlineto{\pgfpoint{\pgf@xb}{\pgf@ya}}
    \pgfpathclose
 }
}
\makeatother

\usepackage{xcolor}

\title[On Decomposing Graphs Into Forests and Pseudoforests]{On Decomposing Graphs Into Forests and Pseudoforests}

\author[Grout]{Logan Grout}
\address[Logan Grout]{Departement of Combinatorics and Optimization, University of Waterloo, Waterloo, ON, Canada}
\email{lcgrout@edu.uwaterloo.ca}

\author[Moore]{Benjamin Moore}
\address[Benjamin Moore]{Department of Combinatorics and Optimization, University of Waterloo, Waterloo, ON, Canada}  
\email{brmoore@uwaterloo.ca}

\date{}

\newtheorem{thm}[equation]{Theorem}
\newtheorem{lemma}[equation]{Lemma}

\newtheorem{conj}[equation]{Conjecture}
\newtheorem{cor}[equation]{Corollary}

\theoremstyle{definition}
\newtheorem{definition}[equation]{Definition}
\newtheorem{obs}[equation]{Observation}

\newtheorem*{ack}{Acknowledgements}

\newtheoremstyle{case}{}{}{\normalfont}{}{\itshape}{\normalfont:}{ }{}

\theoremstyle{case}

\numberwithin{equation}{section}

\date{}

\begin{document}
\begin{abstract}
We prove that for $k \in \mathbb{N}$ and $d \leq 2k+2$, if a graph has maximum average degree at most $2k + \frac{2d}{d+k+1}$, then $G$ decomposes into $k+1$ pseudoforests, where one of the pseudoforests has all connected components having at most $d$ edges.
\end{abstract}

\maketitle
\section{Introduction}

Throughout this paper, all graphs are finite and may contain multiple edges, but have no loops. All undefined graph theory terminology can be found in \cite{BondyAndMurty}. For a graph $G$, $V(G)$ denotes the vertex set and $E(G)$ denotes the edge set. We will let $v(G) = |V(G)|$ and $e(G) = |E(G)|$. Given a graph $G$, an \textit{orientation} of $G$ is obtained from $E(G)$ by taking each edge $xy$, and replacing $xy$ with exactly one of the arcs $(x,y)$ or $(y,x)$. For any vertex $v$, let $d(v)$ and $d^{+}(v)$ denote the degree and out-degree, respectively, of $v$. Let $\Delta(G)$ and $\Delta^{+}(G)$ denote the maximum degree and out-degree respectively.  For any graph $G$, we will say a \textit{decomposition} of $G$ is a set of edge disjoint subgraphs such that the union of their edges sets is $E(G)$. 

 We will focus on decompositions of $G$ into pseudoforests. A \textit{pseudoforest} is a graph where each connected component has at most one cycle. We motivate our results by looking at known results and conjectures about decomposing into forests. 
 
An interesting question is to determine conditions on a graph $G$ to guarantee a decomposition of $G$ into $k$ (pseudo)forests. For forests, this question is answered by the celebrated Nash-Williams Theorem \cite{Nashwilliams}. Before we can state the theorem, we need a definition.

 The \textit{arboricity} of a graph is the minimum number $k$ such that $G$ admits a decomposition into $k$ forests. We will say the \textit{fractional arboricity} of $G$, denoted $\Gamma_{f}(G)$, is 
\[\Gamma_{f}(G) = \max_{H \subseteq G, v(H) >1} \frac{e(H)}{v(H)-1}.\]  Now we state the Nash-Williams Theorem.

\begin{thm}[Nash-Williams, \cite{Nashwilliams}]
\label{nashwilliams}
A graph can be decomposed into $k$ forests if and only if  $\Gamma_{f}(G) \leq k$. 
\end{thm}

Given Theorem \ref{nashwilliams}, a natural question to ask is: assuming we can decompose our graph into $k$ forests, is there a decomposition into $k$ forests where one of the forests has additional structure? There has been a large amount of work in this area, most of which is focused around two conjectures: the Nine Dragon Tree Conjecture, and the Strong Nine Dragon Tree Conjecture. These conjectures were proposed by Montassier, Ossona de Mendez, Raspaund and Zhu \cite{montassier}. The Nine Dragon Tree conjecture was recently proven by Jiang and Yang \cite{ndt}, using techinques from a previous paper of Yang \cite{Yangmatching}.

\begin{thm}[Nine Dragon Tree Theorem, \cite{ndt}]
\label{ndttheorem}
Let $G$ be a graph where $\Gamma_{f}(G) \leq k + \frac{d}{d+k+1}$, then $G$ decomposes into $k+1$ forests, where one of the forests has maximum degree at most $d$. 
\end{thm}

\begin{conj}[Strong Nine Dragon Tree Conjecture]
\label{SNDT}
Let $G$ be a graph where $\Gamma_{f}(G) \leq k+ \frac{d}{d+k+1}$, then $G$ decomposes into $k+1$ forests where one of the forests has the property that every connected component has at most $d$ edges. 
\end{conj}

When $d =1$, the Strong Nine Dragon Tree Conjecture is equivalent to the Nine Dragon Tree Theorem, and hence the Strong Nine Dragon Tree Conjecture is true for any $k \in \mathbb{Z}$ when $d =1$. Kim, Kostochka, West, Wu, and Zhu \cite{kostochkaetal} showed that the Strong Nine Dragon Tree Conjecture is true when $k=1$ and $d=2$. In \cite{montassier}, it was shown that the fractional arboricity bound given in the Nine Dragon Tree Theorem is best possible (and hence also best possible for the Strong Nine Dragon Tree Conjecture). However, we note that in \cite{ndtk<=2}, Chen, Kim, Kostochka, West and Zhu prove a strengthening of the Nine Dragon Tree Theorem when $k \leq 2$ with a weaker fractional arboricity bound but with an extra structural hypothesis to ensure that the graph has a decomposition into $k+1$ forests. 

We are interested in proving similar theorems to this for pseudoforests.
Hakimi's Theorem \cite{hakimithm} is to pseudoforests as Nash-Williams Theorem is to forests, where the maximum average degree parameter arises instead of fractional arboricity. Recall, the \textit{maximum average degree} of $G$ is
\[\text{mad}(G) = \max_{H \subseteq G} \frac{2e(H)}{v(H)}.\]
Observe that a graph is a pseudoforest if and only if the graph admits an orientation such that each vertex has maximum out-degree $1$. Hakimi's Theorem asserts:

\begin{thm}[Hakimi's Theorem \cite{hakimithm}]
\label{hakimi}
A graph $G$ admits an orientation such that $\Delta^{+}(G) \leq k$ if and only if $\text{mad}(G) \leq 2k$.
\end{thm}

Hence we immediately obtain:

\begin{cor}
If $\text{mad}(G) \leq 2k$, then $G$ decomposes into $k$ pseudoforests.
\end{cor}

With this theorem, one might hope to prove a conjecture analagous to the Nine Dragon Tree conjecture for pseudoforests. This was done by Fan, Li, Song, and Yang \cite{pseudoforest}, who showed that:

\begin{thm}[\cite{pseudoforest}]
\label{pseudoforestndt}
If $\text{mad}(G) \leq 2k + \frac{2d}{d+k+1}$, then $G$ decomposes into $k+1$ pseudoforests, where one of the pseudoforests has the property that every connected component has maximum degree $d$. 
\end{thm}

In \cite{pseudoforest}, it was shown that the maximum average degree bound in Theorem \ref{pseudoforestndt} is tight. Analogously to the Strong Nine Dragon conjecture, one can postulate the following:

\begin{conj}
\label{strongpndt}
If $\text{mad}(G) \leq 2k + \frac{2d}{d+k+1}$, then $G$ decomposes into $k+1$ pseudoforests, where one of the pseudoforests has the property that every connected component has at most $d$ edges. 
\end{conj}

By Theorem \ref{pseudoforestndt}, Conjecture \ref{strongpndt} is true when $d=1$ for all $k \in \mathbb{N}$. Interestingly, the proof of the Strong Nine Dragon Tree Theorem in \cite{kostochkaetal} implies Conjecture \ref{strongpndt} when $k=1$ and $d=2$. We show the following theorem:
\begin{thm}
\label{ourresult2}
Conjecture \ref{strongpndt} is true for all tuples $(k,d)$ where $k \in \mathbb{N}$ and $d \leq 2k+2$. 
\end{thm}
 As with the Strong Nine Dragon Tree Theorem, our proof is heavily influenced by Theorem \ref{pseudoforestndt}, and by Theorem \ref{ndttheorem}.

 As an overview of the paper, Section 2 proves Theorem \ref{ourresult2}. A high level overview of the proof is the following. The proof is by contradiction, where we will take a counterexample subject to many minimality conditions. We start with a decomposition into $k+1$ pseudoforests, $(C_{1},\ldots,C_{k},F)$, where we think of $F$ as the special pseudoforest which we want to have small component sizes (we will pick this starting decomposition in a particular way, it will not be an arbitrary decomposition). As we are working with a counterexample, we can assume that $F$ has a component with more than $d$ edges. For $i \in \{1,\ldots,k\}$, we colour the edges of $C_{i}$ blue, and we colour the edges of $F$ red. We then pick a component of $F$ which has more than $d$ edges and designate it as a root component (subject to some minimality conditions). Using the fact that pseudoforests admit a natural orientation, we can orient $C_{1},\ldots,C_{k}$ and use this orientation to order the red components of $F$ (and we will pick this ordering subject to some conditions). The rest of the proof is to show that in this ordering, if a component of $F$ has few edges, it must be paired up with components of $F$ which have a large number of edges, else we could perform a ``flip" which would contract our choice of counterexample. This pairing will show that we must have maximum average degree larger than $2k + \frac{2d}{k+d+1}$, a contradiction.

\section{Decomposing sparse graphs into pseudoforests}

We split the proof up into four parts. The first part is dedicated to describing how we will choose our counterexample. The second part describes a flipping operation which we will use to try and ``improve" our counterexample. The third part describes the key structural properties of the counterexample. The final part is a counting argument which shows that assuming we could not improve our counterexample, there exists a subgraph with average degree larger than the bound in Theorem \ref{ourresult2}.

\subsection{Picking the minimal counterexample} \mbox{}
Fix $k \in \mathbb{N}$, and $d \leq 2k+2$. We will assume $d \geq 2$, as the case where $d =1$ is already known. Suppose $G$ is a vertex minimal counterexample to Theorem \ref{ourresult2} for those values of $k$ and $d$. 
 
In the proof of Theorem \ref{pseudoforestndt}, the first step was to show that vertex minimal counterexamples admit orientations where the difference between each vertices out-degree differs by at most one, and each vertex has out-degree at least $k$. The existence of these orientations also carries over to minimal counterexamples for Theorem \ref{ourresult2}, by the exact same proof (despite the slightly different hypothesis). We omit the proof.  

\begin{lemma}[\cite{pseudoforest}]
\label{orientationlemma}
If $G$ is a vertex minimal counterexample to Theorem \ref{ourresult2}, then there exists an orientation of $G$ such that for all $v \in V(G)$, we have $k \leq d^{+}(v) \leq k+1$. 
\end{lemma}

 Let $\mathcal{F}$ be the set of orientations of $E(G)$ which satisfy Lemma \ref{orientationlemma}. The first observation is that any orientation in $\mathcal{F}$ gives rise to a natural pseudoforest decomposition. 
 
\begin{definition}
  Let $\sigma \in \mathcal{F}$. Given $\sigma$, a \textit{red-blue colouring of $G$} is a colouring of the edges obtained in the following way. For all $v \in V(G)$, if $d^{+}(v) = k+1$, arbitrarily colour $k$ of the outgoing arcs blue and one of the arcs red. If $d^{+}(v) = k$, then colour all outgoing arcs blue.
\end{definition}

Note that given an orientation in $\mathcal{F}$, one can generate many different red-blue colourings. Recall that if a graph admits an orientation where each vertex has out-degree at most one, then the graph is a pseudoforest. Hence we have the following observation.

\begin{obs}
\label{redbluecolouringdecomp}
Given a red-blue colouring of $G$, we can decompose our graph $G$ into $(k+1)$-pseudoforests such that $k$ of the pseudoforests have all of their edges coloured blue, and the other pseudoforest has all of its edges coloured red. 
\end{obs}

Observe that one red-blue colouring can give rise to many different pseudoforest decompositions (when $k \geq 2$). Given a pseudoforest decomposition obtained from Observation \ref{redbluecolouringdecomp} we will say a pseudoforest which has all arcs coloured blue is a \textit{blue pseudoforest}, and the pseudoforest with all arcs coloured red is the \textit{red pseudoforest}.

\begin{definition}
Let $f$ be a red-blue colouring of $G$, and let $C_{1},\ldots,C_{k},F$ be a pseudoforest decomposition obtained from $f$ by Observation \ref{redbluecolouringdecomp}. Then we say that \textit{$C_{1},\ldots,C_{k},F$ is a pseudoforest decomposition generated from $f$}. We will always use the convention that $F$ is the red pseudoforest, and $C_{i}$ is a blue pseudoforest. 
\end{definition}

  As $G$ is a counterexample, in every pseudoforest decomposition generated from a red-blue colouring, there is a connected component of the red pseudoforest which has more than $d$ edges. 

Now we define a residue function which we will use to pick our minimal counterexample. 
\begin{definition}
Let $f$ be a red-blue colouring and $C_{1},\ldots,C_{k},F$ be a pseudoforest decomposition generated by $f$. Let $\mathcal{T}$ be the set of components $F$. We define the \textit{residue function}, $\rho$, as:
\[\rho(F) = \sum_{K \in \mathcal{T}} \max\{e(K) -d,0\}.\]
\end{definition}

Using a red-blue colouring, and the resulting pseudoforest decomposition, we define an induced subgraph of $G$ which we will focus our attention on. 

\begin{definition}
Suppose that $f$ is a red-blue colouring of $G$, and suppose $D= (C_{1},\ldots,C_{k},F)$ is a pseudoforest decomposition generated from $f$. Let $R$ be an component of $F$ such that $e(R) >d$. We define the subgraph $H_{f,D,R}$ in the following manner. Let $S \subseteq V(G)$ where $v \in S$ if and only if there exists a path $P = v_{1},\ldots,v_{m}$ such that $v_{m}= v$ where $v_{1} \in V(R)$, and either $v_{i}v_{i+1}$ is an arc $(v_{i},v_{i+1})$ coloured blue, or $v_{i}v_{i+1}$ is an arbitrarily directed arc coloured red. Then we let $H_{f,D,R}$ be the graph induced by $S$.
\end{definition}

Given a graph $H_{f,D,R}$, we will say $R$ is the \textit{root}. We will say the \textit{red components} of $H_{f,D,R}$ are the components of $F$ contained in $H_{f,D,R}$, and given a vertex $x \in V(H_{f,D,R})$, we let $R^{x}$ denote the red component of $H_{f,D,R}$ containing $x$.  As notation, given a subgraph $K$ of $G$, we will let $E_{b}(K)$ and $E_{r}(K)$ denote the set of edges of $K$ coloured blue and red, respectively. We will also let $e_{b}(K) = |E_{b}(K)|$ and $e_{r}(K) = |E_{r}(K)|$. 

We make an important observation about the average degree of $H_{f,D,R}$.
By definition, for every vertex $v \in V(H_{f,D,R})$, all blue outgoing arcs from $v$ are in $E(H_{f,D,R})$. Therefore,
\[\frac{e_{b}(H_{f,D,R})}{v(H_{f,D,R})} = k.\]

Hence if the following inequality holds
\[\frac{e_{r}(H_{f,D,R})}{v(H_{f,D,R})} > \frac{d}{d+k+1},\]
we would have
\[\frac{e(H_{f,D,R})}{v(H_{f,D,R})} = \frac{e_{r}(H_{f,D,R})}{v(H_{f,D,R})} + \frac{e_{b}(H_{f,D,R})}{v(H_{f,D,R})} > k + \frac{d}{d+k+1},\]
which would contradict the maximum average degree bound. Our goal will be to pick a particular red-blue colouring $f$ and root component $R$ to obtain such a contradiction. 

Towards this goal, we define the notion of a \textit{troublesome component}.

\begin{definition}
Let $P$ be the red pseudoforest of some red-blue colouring. Let $K$ be a subgraph of $P$. We will say $K$ is \textit{troublesome} if  
\[\frac{e_{r}(K)}{v(K)} < \frac{d}{d+k+1}.\]
\end{definition}

Observe that as $d \leq 2k+2$, it follows that:
\[\frac{d}{d+k+1} \leq \frac{2}{3}.\]

We pause to consider the special case where $K$ is a connected component of the red pseudoforest and troublesome. If $K$ contains a cycle, then $K$ is a pseudoforest, and hence not troublesome. Therefore $K$ must be a tree. It is easy to see that if $K$ is a tree and troublesome, then $K$ must either be a single vertex, or isomorphic to $K_{2}$.

We now define the notion of a legal order of the components of the red pseudoforest in $H_{f,D,R}$. 

\begin{definition}
  We call an ordering $(R_{1},\ldots,R_{t})$ of the red components of $H_{f,D,R}$ \textit{legal} if all components are in the ordering, $R_{1}$ is the root component, and for all $j \in \{2,\ldots,t\}$ there exists an integer $i$ with $1 \leq i < j$ such that there is a blue arc $(u,v)$ such that $u \in V(R_{i})$ and $v \in V(R_{j})$. 
\end{definition} 

Let $(R_{1},\ldots,R_{t})$ be a legal ordering. We will say that $R_{i}$ is a \textit{parent} of $R_{j}$ if $i < j$ and there is a blue arc $(v_{i},v_{j})$ where $v_{i} \in R_{i}$  and $v_{j} \in R_{j}$. In this definition a red component may have many parents. To remedy this, if a red component has multiple parents, we arbitrarily pick one such red component and designate it as the only parent. If $R_{i}$ is the parent of $R_{j}$, then we say that $R_{j}$ is a \textit{child} of $R_{i}$. We say a red component $R_{i}$ is an \textit{ancestor} of $R_{j}$ if we can find a sequence of red components $R_{i_{1}},\ldots,R_{i_{m}}$ such that $R_{i_{1}} = R_{i}$, $R_{j_{m}} = R_{j}$, and $R_{i_{q}}$ is the parent of $R_{i_{q+1}}$ for all $q \in \{1,\ldots,m-1\}$. 

An important definition we need is the notion of vertices determining a legal order.
\begin{definition}
Given a legal order $(R_{1},\ldots,R_{t})$, we will say a vertex $v$ \textit{determines the legal order for $R_{j}$} if there is a blue arc $(u,v)$ such that $u \in R_{i}$ and $v \in R_{j}$ such that $i < j$. 
\end{definition}
  Observe there may be many vertices which determine the legal order for a given red component. More importantly, for every component which is not the root, there exists a vertex which determines the legal order.
  
We also want to compare two different legal orders. 

\begin{definition}
Let $(R_{1},\ldots,R_{t})$ and $(R'_{1},\ldots,R'_{t'})$ be two legal orders. We will say $(R_{1},\ldots,R_{t})$ is \textit{smaller} than $(R'_{1},\ldots,R'_{t})$ if  the sequence $(e(R_{1}),\ldots,e(R_{t}))$ is smaller lexicographically than $(e(R'_{1}),\ldots,e(R'_{t'}))$. 
  
\end{definition}

With this, we will pick our minimal counterexample in the following manner.

First, we pick our counterexample to be minimized with respect to the number of vertices. Then, we pick an orientation in $\mathcal{F}$, a red-blue colouring $f$ of this orientation, a pseudoforest decomposition $D = (C_{1},\ldots,C_{k},F)$ generated by $f$, a root component $R$ of $F$, and lastly a legal order $(R_{1},\ldots,R_{t})$ of $H_{f,D,R}$ subject to the following conditions in the following order:

\begin{enumerate}
\item{The number of cycles in $F$ is minimized (so the number of cycles coloured red),}

\item{subject to the above condition, we minimize the residue function $\rho$,}

\item{and subject to both of the above conditions, $(R_{1},\ldots,R_{t})$ is the smallest legal order.}

\end{enumerate}

From here on out, we will assume we are working with a counterexample picked in the manner described.

\subsection{The flip operation}
Let $f$ be the red-blue colouring of our counterexample, and let $C_{1},\ldots,C_{k}$ be the blue pseudoforests, and $F$ the red pseudoforest. Let $(R_{1},\ldots,R_{t})$ be the legal ordering picked for our counterexample.

\begin{definition}
Let $f$ be a red-blue colouring. Let $(x,y)$ be an arc coloured blue, $y$ is not in a red cycle, and suppose that $e = xv$ is a arbitrarily oriented red arc incident to $x$. To \textit{flip on $e$ and $(x,y)$} is to take a maximal directed red path $Q = v_{1},v_{2},\ldots,v_{n},y$ where $(v_{i},v_{i+1})$ is a red arc, $(v_{n},y)$ is a red arc, reverse the direction of all arcs of $P$, change the colour of $(x,y)$ to red, reorient $(x,y)$ to $(y,x)$, change the colour of $e$ to blue, and (if necessary), reorient $(v,x)$ to $(x,v)$. 
\end{definition}

\begin{obs}
\label{flipobservation}
Suppose we flip on an edge $e =xv$ and $(x,y)$. Then the resulting orientation is in $\mathcal{F}$. 
\end{obs}

\begin{proof}
If $xv$ is oriented $(x,v)$, then the flip operation changes the colour of $(x,y)$ to red and the colour of $e$ to blue, all internal vertices of $Q$ still have red out-degree at most one, $y$ has red out-degree one as we orient $(x,y)$ to $(y,x)$, and $v$ still has red out-degree at most one. By construction all the blue out-degrees stay the same. Hence this orientation is in $\mathcal{F}$. 

If $xv$ is oriented $(v,x)$, then the flip operation changes the colour of $(x,y)$ to red, and then reorients $(v,x)$ to $(x,v)$ and reverses the orientation on a maximal directed red path $Q$. By construction, the blue out-degrees of all vertices remain the same, and similarly, the red out-degrees stay the same. Hence this orientation is in $\mathcal{F}$.   
\end{proof}

To avoid repetitively mentioning it, we will implicitly make use of Observation \ref{flipobservation}. 

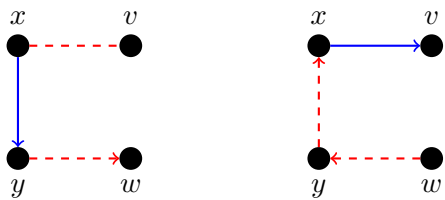
\begin{figure}
\label{flip}
\centering
\begin{tikzpicture}
\node[blackvertex, label = {$x$}] at (0,0) (v1) {$x$};
\node[blackvertex, label = {$v$}] at (1.5,0) (v2) {$x$};
\node[blackvertex, label = below:{$y$}] at (0,-1.5)(v3) {$x$};
\node[blackvertex, label = below:{$w$}] at (1.5,-1.5) (v4) {$x$};
\draw[thick,dashed,red] (v1)--(v2);
\draw[thick,blue, ->] (v1)--(v3);
\draw[thick,dashed,red,->] (v3)--(v4);

\begin{scope}[xshift = 4cm]
\node[blackvertex,label = {$x$}] at (0,0) (v1) {$x$};
\node[blackvertex,label = {$v$}] at (1.5,0) (v2) {$x$};
\node[blackvertex,label = below:{$y$}] at (0,-1.5)(v3) {$x$};
\node[blackvertex,label = below:{$w$}] at (1.5,-1.5) (v4) {$x$};
\draw[thick,blue, ->] (v1)--(v2);
\draw[thick,dashed,red, ->] (v3)--(v1);
\draw[thick,dashed,red,->] (v4)--(v3);
\end{scope}
\end{tikzpicture}
\caption{An example of a flip on an edge $e=xv$ and $(x,y)$. Dashed lines indicate red edges, and solid lines are blue edges. The path $y,w$ is $Q$ in this example.}
\end{figure}

\subsection{Key Lemmas}
Now we are ready to prove the main lemmas. First we make a basic observation about red cycles.

\begin{obs}
\label{redcyclesaturationlemma}
Let $(x,y)$ be an edge which is coloured blue such that $R^{x}$ is distinct from $R^{y}$ and $R^{y}$ is a tree. Then $x$ does not lie in a cycle of $F$. 
\end{obs}

\begin{proof}
Suppose towards a contradiction that $x$ lies in a cycle of $F$. Let $e$ be an edge incident to $x$ which lies in the cycle coloured red. Now flip at $(x,y)$ and $e$. As $(x,y)$ was an arc between two red components, and $e$ was in the cycle coloured red, after performing the flip, we reduce the number cycle in $F$ by one. However, this contradicts our choice of minimum counterexample.  
\end{proof}

Now we show that for any red component and any child of that red component either the child contains a cycle, or together they have at least $d$ edges.


\begin{lemma}
\label{troublesomedonttroublesome}
Let $R^{x}$ and $R^{y}$ be red components such that $R^{y}$ is the child of $R^{x}$, $R^{y}$ does not contain a cycle, and $(x,y)$ is a blue arc from $x$ to $y$. Then $e(R^{x}) + e(R^{y}) \geq d$. 
\end{lemma}

\begin{proof}
Suppose towards a contradiction that $e(R^{x}) + e(R^{y}) < d$. Hence $e(R^{x}) < d$. Thus $R^{x}$ is not the root component.  Let $w$ be a vertex which determines the legal order for $R^{x}$. Observe that by Observation \ref{redcyclesaturationlemma} $x$ does not belong to a cycle of $R^{x}$.

\textbf{Case 1:} $w \neq x$.

Let $e$ be the edge incident to $x$ in $R^{x}$ such that $e$ lies on the path from $x$ to $w$ in $R^{x}$. Then flip on $(x,y)$ and $e$. As $e(R^{x}) + e(R^{y}) < d$, all resulting red components have less than $d$ edges, and hence we do not increase the residue function. Furthermore, we claim we can find a smaller legal order. Let $R_{i}$ be the component in the legal order corresponding to $R^{x}$. Then consider the new legal order where the components $R_{1},\ldots,R_{i-1}$ remain in the same position, we replace $R_{i}$ with $R^{w}$, and then complete the order arbitrarily. By how we picked $e$, $e(R^{w})$ is strictly smaller than $e(R_{i})$, and hence we have found a smaller legal order, a contradiction. 

\textbf{Case 2:} $w =x$.

As $R^{x}$ is not the root component, let $R^{x_{n}}$ be an ancestor of $R^{x}$ such that $e(R^{x_{n}}) \geq 1$. Let $R^{x_{n}},R^{x_{n-1}},\ldots,R^{x_{1}}$ be a sequence of red components such that for $i \in \{1,\ldots,n-1\}$, $R^{x_{i}}$ is the child of $R^{x_{i+1}}$ and $R^{x}$ is the child of $R^{x_{1}}$. Up to relabelling the vertices,  there is a path $P = x_{n},\ldots,x_{1},x$ such that for all $i \in \{1,\ldots,n-1\}$,  $(x_{i+1},x_{i})$ is an arc coloured blue, and $(x_{1},x)$ is a arc coloured blue. Let $e$ be a red edge incident to $x_{n}$. Now do the following. Colour $(x,y)$ red, and reverse the direction of all arcs in $P$. Colour $e$ blue, and orient $e$ away from $x_{n}$. By the same argument in Observation \ref{flipobservation}, the resulting orientation is in $\mathcal{F}$. Furthermore as $e(R^{x}) + e(R^{y}) < d$, all resulting red components have less than $d$ edges, and hence $\rho$ did not increase. Finally, we can find a smaller legal order in this orientation, as we simply take the same legal order up to the original component containing $x_{n}$, and then complete the remaining order arbitrarily. As the component containing $x_{n}$ has at least one less edge now, this order is smaller lexiographically, a contradiction.

\end{proof}

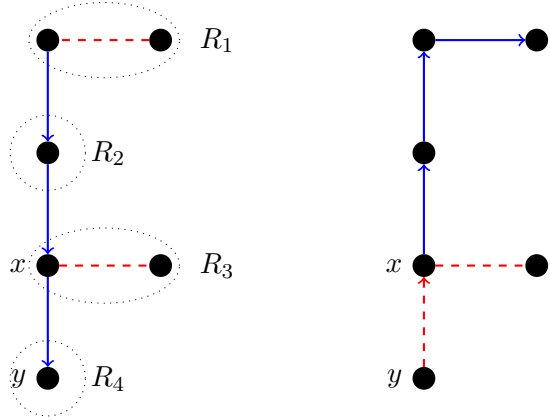
\begin{figure}
\label{flip}
\centering
\begin{tikzpicture}
\node[blackvertex, label = left:{$x$}] at (0,0) (v1) {$x$};
\node[blackvertex] at (1.5,0) (v2) {$x$};
\draw[dashed,thick,red] (v1)--(v2);
\draw[dotted]  (.75,0) ellipse(1cm and .5cm);
\node[dummywhite] at (2.25,0) (v10) {$R_{3}$};

\node[blackvertex, label = left: {$y$}] at (0,-1.5) (v3) {$x$};
\draw[thick,blue,->] (v1)--(v3);
\draw[dotted] (0,-1.5) ellipse (0.5cm and 0.5cm);
\node[dummywhite] at (.8,-1.5) (v10) {$R_{4}$};

\node[blackvertex] at (0,1.5) (v4) {$x$};
\node[blackvertex] at (0,3) (v5) {$x$};
\draw[dotted] (0,1.5) ellipse (0.5cm and 0.5cm);
\node[dummywhite] at (.8,1.5) (v10) {$R_{2}$};
\draw[dotted] (.75,3) ellipse (1cm and 0.5cm);
\node[dummywhite] at (2.25,3) (v10) {$R_{1}$};
\draw[thick,blue,->] (v5)--(v4);
\draw[thick,blue,->] (v4)--(v1);

\node[blackvertex] at (1.5,3) (v6) {$x$};
\draw[thick,red,dashed] (v6)--(v5);

\begin{scope}[xshift = 5cm]
\node[blackvertex, label = left:{$x$}] at (0,0) (v1) {$x$};
\node[blackvertex] at (1.5,0) (v2) {$x$};
\draw[dashed,thick,red] (v1)--(v2);

\node[blackvertex, label = left: {$y$}] at (0,-1.5) (v3) {$x$};
\draw[thick,red,dashed,->] (v3)--(v1);

\node[blackvertex] at (0,1.5) (v4) {$x$};
\node[blackvertex] at (0,3) (v5) {$x$};

\draw[thick,blue,->] (v4)--(v5);
\draw[thick,blue,->] (v1)--(v4);

\node[blackvertex] at (1.5,3) (v6) {$x$};
\draw[thick,blue,->] (v5)--(v6);
\end{scope}
\end{tikzpicture}
\caption{An example of a possible situation in Case $2$ of Lemma \ref{troublesomedonttroublesome}. Here in the picture on the left we have components $R_{1},R_{2},R_{3},R_{4}$, where they also appear in that legal order,  $R_{2}$ is an isolated vertex, and $R_{1}$ possibly has more red edges left undrawn. The picture on the right is the operation done in Case $2$. }
\end{figure}

To avoid repetition in the next argument, we make the following observation.

\begin{obs}
\label{twolargecomponents}
Let $(x,y)$ be a blue arc such that $R^{y}$ is a troublesome child of $R^{x}$. Then there does not exist an edge $e \in E(R^{x})$ incident to $x$ such that flipping on $(x,y)$ and $e$ results in two new components, both with more than $d$ edges. 
\end{obs}

\begin{proof}
Suppose towards a contradiction that such an edge $e$ exists. Flip on $(x,y)$ and $e$. Let $R^{x'}$ and $R^{x''}$ be the new components. Therefore, we increase the residue function by $e(R^{x'}) + e(R^{x''}) -2d$ and decrease the residue function by $e(R^{x}) -d$. Observe that 
\[e(R^{x'}) + e(R^{x''}) = e(R^{y}) + e(R^{x}),\] and $e(R^{y}) \leq 1 < d$. Hence \[e(R^{x'}) + e(R^{x''}) -2d = e(R^{y}) + e(R^{x}) -2d \leq e(R^{x}) -d.\]
Hence the residue function decreases, contradicting our choice of minimal counterexample. 
\end{proof}

Now we show that any red component with at least two edges has few troublesome children.

\begin{lemma}
\label{boundedchildren}
	Let $R^{x}$ be a red component where if $\frac{d}{d+k+1} < \frac{1}{2}$, then  $e(R^{x}) \geq 2$, and otherwise $e(R^{x}) \geq 3$. Then $R^{x}$ has at most $k$ troublesome children. 
\end{lemma}

\begin{proof}

We will prove a statement slightly stronger than the conclusion in the lemma. We will show that there is a vertex $v \in R^{x}$ such that for every troublesome child $R$, there is a blue arc from $v$ to $R$. This implies the lemma statement, as every vertex has at most $k$ blue outgoing arcs.

Suppose towards a contradiction that $R^{x}$ has more than $k$ troublesome children. Then as every vertex has at most $k$ blue outgoing arcs, there are two vertices, say $x$ and $w$, which have blue coloured arcs to troublesome children. As before, by Observation \ref{redcyclesaturationlemma}, $x$ and $w$ do not lie in a cycle in $R^{x}$. 

We separate into cases based on if $R^{x}$ is the root component or not. 

\textbf{Case 1:} $R^{x}$ is not the root.

Without loss of generality, we may assume that there is a vertex other than $x$ which determines the legal order for $R^{x}$, call this vertex $z$. 

\textbf{Subcase 1:} $d_{r}(z) \geq 2$. 

Let $(x,y)$ be an arc coloured blue such that $R^{y}$ is a troublesome child of $R^{x}$. Then let $e$ be the edge incident to $x$ which lies on any $(x,z)$ path in $R^{x}$. Then flip on $e$ and $(x,y)$. Note that since $e(R^{y}) \leq 1$ and  $d_{r}(z) \geq 2$, the resulting component containing $x$ has at most as many edges as $R^{x}$. Hence the residue function does not increase. Furthermore, we can find a smaller legal order, simply by taking the same legal order up until $R^{x}$, and then replacing $R^{x}$ with the component containing $z$, and then completing the order arbitrarily. This new legal order is smaller, as the component containing $z$ has strictly fewer edges than $R^{x}$, contradicting our choice of minimal counterexample.   

\textbf{Subcase 2:} $d_{r}(z) = 1$

There are two possibilities, either $z$ could have an arc coloured blue, say $(z,q)$, such that $R^{q}$ is a troublesome child of $R^{x}$ or there is no such arc. First suppose that $z$ does not have such an arc. Let $e$ be the edge incident to $x$ on any $(x,z)$ path in $R^{x}$, and $e'$ be the edge incident to $w$ on any $(w,z)$ path in $R^{x}$. Then either the component containing $z$ in $R^{x} -e$ has at least one edge, or the component containing $z$ in $R^{x} -e'$ has at least one edge. Without loss of generality we can assume that the component containing $z$ in $R^{x} -e$ has at least one edge. Then perform a flip on $(x,y)$ and $e$. It follows by the same argument as in Subcase 1 that this contradicts our choice of counterexample. 

Now suppose that there is an arc $(z,q)$ such that $R^{q}$ is a troublesome child of $R^{x}$. Then let $e$ be the edge incident to $z$ in $R^{x}$. Flip on $e$ and $(z,q)$. Note that the resulting component containing $z$ has fewer edges, as $R^{q}$ is troublesome, and we assume that if $\frac{d}{d+k+1} < \frac{1}{2}$ that $e(R^{x}) \geq 2$, and otherwise $e(R^{x}) \geq 3$. Furthermore, we did not increase the residue function, as the component of $R^{x} -e$ not containing $z$ has fewer edges, and after flipping the component containing $z$ has at most $d$ edges. Now we find a smaller legal order by taking the same ordering up to $R^{x}$, and then replacing $R^{x}$ with the component containing $z$, and the complete the order arbitrarily. By the above observations, this is a smaller legal order, contradicting our choice of counterexample.

\textbf{Case 2:} $R^{x}$ is the root component.

In this case, we split into cases based on the red degree of $w$, and follow effectively the same analysis as above. 

\textbf{Subcase 1:} $d_{r}(w) = 1$

Let $(w,q)$ be a blue arc such that $R^{q}$ is a troublesome child of $R^{x}$. 
Let $e$ be the edge incident to $w$ in $R^{x}$. Flip on $e$ and $(w,q)$. The resulting component containing $w$ has less than $d$ edges. Furthermore, this reduces the residue function as now the component of $R^{x} -e$ not containing $w$ has fewer edges, and $e(R^{x}) >d$ to begin with, as $R^{x}$ is the root. Hence we have a contradiction. 

\textbf{Subcase 2:} $d_{r}(w) \geq 2$. 

Now let $e$ be the edge incident to $x$ on any $(x,w)$-path in $R^{x}$. Flip on $(x,y)$ and $e$. Now the new component containing $x$ has at most the same number of edges as $R$. If this component has strictly fewer edges, then we have decreased the residue function, a contradiction. Therefore we can assume the component containing $x$ after the flip has more than $d$ edges. Hence the component of $R^{x} - e$ containing $x$ has at least $d -1$ edges. Now let $e'$ be the edge incident to $w$ on any $(w,x)$-path in $R^{x}$. Flip on $(w,q)$ and $e$. Either Observation \ref{twolargecomponents} applies, or we reduce the residue function as the root component now is strictly smaller. In either case, we obtain a contradiction. 
\end{proof}
\subsection{Pairing up the red components} \mbox{}
 For ease of notation, we make the following definition. 

\begin{definition}
 Let $T$ be a red component which is not troublesome, and suppose that $T_{1},\ldots,T_{j}$ are troublesome children of $T$. We define $T_{C}$ to be the subgraph with vertex set $V(T_{C}) = V(T) \cup V(T_{1}) \cup \cdots \cup V(T_{j})$, and the edge set is all edges coloured red on that vertex set.  
\end{definition} 
 
 Now we prove the final ingredient necessary for the proof.

\begin{lemma}
\label{pairinglemma}
Let $T$ be a red component which is not troublesome with $T_{1},\ldots,T_{j}$ as troublesome children. Then $T_{C}$ is not troublesome. Furthermore, if $e(T) >d$, then 
\[\frac{e(T_{C})}{v(T_{C})} > \frac{d}{k+d+1}.\]
\end{lemma}

\begin{proof}
Suppose towards a contradiction that $T_{C}$ is troublesome. By Lemma \ref{troublesomedonttroublesome}, we may assume that for all $i \in \{1,\ldots,j\}$, the inequality $e(T) + e(T_{i}) \geq d$ holds. If $\frac{d}{d+k+1} < \frac{1}{2}$, we have that $e(T_{i}) =0$,  so  $e(T) \geq 2$. Otherwise $e(T_{i}) \leq 1$, and again it follows that $e(T) \geq 2$. By Lemma \ref{boundedchildren}, $j \leq k$. We consider cases. 

\textbf{Case 1:} $e(T) \geq d$. 

In this case we have
\begin{align*}
\frac{e(T_{C})}{v(T_{C})} &= \frac{e(T) + \sum_{i=1}^{j} e(T_{i})}{v(T) + \sum_{i=1}^{j}v(T_{i})} 
\\ &= \frac{e(T) + \sum_{i=1}^{j} (v(T_{i}) -1)}{v(T) + \sum_{i=1}^{j}v(T_{i})}
\\&\geq \frac{e(T)}{v(T) +j}  
\\ & \geq \frac{e(T)}{v(T) +k}
\\ & \geq \frac{d}{k+d+1}.
\end{align*} 
The first line follows immediately from the definition of $T_{C}$. The second line follows since for all $i$, $T_{i}$ is a tree. One can check that the third line holds by expanding and simplifying the inequality. The fourth line follows as $j \leq k$. The last line follows by noticing that if $T$ is a pseudoforest component the result holds, and that if $T$ is a tree, using $e(T) = v(T) -1$ and $e(T) \geq d$ gives the claim. 

We observe that the last inequality becomes strict if $e(T) >d$. This proves the second part of the lemma. Observe that this also proves the claim when $d =2$.

\textbf{Case 2:} $e(T) <d$.

Note that $e(T) + e(T_{i}) \geq d$ (by Lemma \ref{troublesomedonttroublesome}) and that $e(T_{i}) \leq 1$. Therefore $e(T) \geq d-1$. Then as $e(T) < d$, we have $e(T_{i}) =1$ (here we are using that $d \geq 2$). Then we have:
\begin{align*}
\frac{e(T_{C})}{v(T_{C})} &= \frac{e(T) + \sum_{i=1}^{j} e(T_{i})}{v(T) + \sum_{i=1}^{j}v(T_{i})}
\\ & = \frac{e(T) + \sum_{i=1}^{j}(v(T_{i}) -1)}{v(T) + \sum_{i=1}^{j}v(T_{i})}
\\ &\geq \frac{e(T) +j}{v(T) +2j}  
\\ &\geq \frac{d + j -1}{d+2j}
\\ &\geq \frac{d+ k-1}{d+2k}
\\ &\geq \frac{d}{d+k+1}.
\end{align*}
These inequalities hold for essentially the same reasons as in Case 1. The lemma follows. 
\end{proof}

Now we finish the proof.  Let $\mathcal{R}$ denote the set of red components which are not troublesome. By Lemma \ref{troublesomedonttroublesome} it follows that,
\[V(H) = \bigcup_{R \in \mathcal{R}} V(R_{C}).\]

This follows since a troublesome component cannot have a troublesome child by  Lemma \ref{troublesomedonttroublesome}. Therefore it follows that:
\[E_{r}(H_{f,D,R}) = \bigcup_{R \in \mathcal{R}} E(R_{C}).\]
Since $R_{C}$ is not troublesome for any $R$ by Lemma \ref{pairinglemma}, and they partition $H_{f,D,R}$, we have 
\[\frac{e_{r}(H_{f,D,R})}{v(H_{f,D,R})} \geq \frac{d}{d+k+1}.\]

 Furthermore, the root component $R$ satisfies $e(R) >d$, hence,
 \[\frac{e_{r}(H_{f,D,R})}{v(H_{f,D,R})} >\frac{d}{d+k+1}.\]
 
 Therefore,
 \[\frac{e(H_{f,D,R})}{v(H_{f,D,R})} = \frac{e_{r}(H_{f,D,R})}{v(H_{f,D,R})} + \frac{e_{b}(H_{f,D,R})}{v(H_{f,D,R})} > k + \frac{d}{d+k+1},\]
 
which contradicts the maximum average degree bound and hence Theorem \ref{ourresult2} is true.

\begin{ack}
The authors would like to acknowledge Joseph Cheriyan for numerous stimulating discussions about the topic. We would also like to acknowledge the anonymous referees which greatly helped improve the presentation of the paper, and pointed out a strengthening of our initial results.  
\end{ack}

\bibliographystyle{plain}
\bibliography{pseudo}

\end{document}